\documentclass[12pt, a4paper, twoside, reqno]{amsart}

\usepackage[margin=2.85cm]{geometry}

\usepackage[T1]{fontenc}
\usepackage[latin1]{inputenc}
\usepackage{amsmath}
\usepackage{amsthm}
\usepackage{amsfonts}
\usepackage{MnSymbol}
\usepackage{bbm}
\usepackage[UKenglish]{babel}
\usepackage{enumerate}
\usepackage[dvips]{graphicx}
\usepackage{bm}
\usepackage{ae}

\usepackage[draft,english]{fixme}

\theoremstyle{definition}

\theoremstyle{remark}

\theoremstyle{plain}
\newtheorem{thm}{Theorem}
\newtheorem{lem}[thm]{Lemma}
\newtheorem{prop}[thm]{Proposition}
\newtheorem{cor}[thm]{Corollary}

 \DeclareMathOperator{\Coeff}{Coeff}
 \DeclareMathOperator{\MDS}{\mathbf{MDS}}

 \newcommand{\NN}{\mathbb{N}}
 \newcommand{\ZZ}{\mathbb{Z}}
 \newcommand{\QQ}{\mathbb{Q}}
 \newcommand{\RR}{\mathbb{R}}

 \newcommand{\abs}[1]{\left\lvert #1 \right\rvert}
 
 \newcommand{\inv}[1][1]{^{- #1}}
 \newcommand{\set}[1]{\left\lbrace #1 \right\rbrace}
 \newcommand{\card}{\#}
 
 \newcommand{\FF}{\mathbb{F}_3}
 \newcommand{\FFq}{\mathbb{F}_q}
 \newcommand{\HH}{\mathcal{H}}
 \newcommand{\CC}{\mathcal{C}}
 \newcommand{\II}{\mathbb{I}}
 \newcommand{\MDSA}{\MDS(\mathcal{A})}

\begin{document}

\title{A Cantor set type result in the field of formal Laurent series.}

\author{S. H. PEDERSEN}

\address{S. H. PEDERSEN, Department of Mathematics, Aarhus University,
  Ny Munkegade 118, DK-8000 Aarhus C, Denmark}

\email{steffenh@imf.au.dk}



\begin{abstract}
	We prove a Khintchine type theorem for approximation of elements in the Cantor set, as a subset of the formal Laurent series over $\FF$, by rational functions of a specific type.
	
	Furthermore we construct elements in the Cantor set with any prescribed irrationality exponent $\geq 2$.
\end{abstract}

\maketitle

\section{Introduction}
\label{sec:introduction}

In \cite{MR1512207} Khintchine proved, that for $\psi:\RR_{\geq 1} \to \RR_{> 0}$ a continuous function with \break$x \mapsto x^2 \psi(x)$ non-increasing, the set
    \[
        W(\psi) = \set{ \xi \in \RR :
             \abs{ \xi - \dfrac{p}{q} } < \psi(q)
             \text{ for infinitely many } \dfrac{p}{q} \in \QQ}
    \]
of $\psi$-well approximable numbers has Lebesgue measure $0$ if the series
    \[
        \sum_{q = 1}^{\infty} q\psi(q)   
    \]
converges, and full Lebesgue measure if the series diverge. The analogues statement in the field of formal Laurent series over finite fields was shown by de Mathan in \cite{MR0274396}.

In \cite{MR2295506} Levesly, Salp and Velani established a Khintchine type theorem for $\psi$-well approximable numbers in the Cantor set by rational numbers of the form $\frac{p}{3^n}, p \in \NN$.

The first part of this paper will establish the analogous statement in the field of formal Laurent series over $\FF$, where the Cantor set consists of those formal Laurent series in the unit ball around $0$ having only the coefficients $0$ and $2$.

The second part of the paper will construct elements of the Cantor set with any prescribed irrationality exponent $\geq 2$. This is the analogue of the result in \cite{MR2399165} by Bugeaud.

The proofs follow the approach from \cite{MR2295506} and \cite{MR2399165}, but with the simplifications and complications of working over an ultrametric field.

\section{Preliminaries}
\label{sec:preliminaries}

Let $\FF$ be the field with $3$ elements and let $\FF[X]$ be the polynomial ring over $\FF$. We can introduce an absolute value on $\FF[X]$, by letting $\abs{P}=3^{\deg P}$ for $P \in \FF[X] \setminus \set{0}$, and $\abs{0}=0$. This in turn gives an absolute value on the rational functions $\FF(X)$, and by completing with respect to this absolute value, we get the field of formal Laurent series over $\FF$, that is the set
    \[
        \FF((X\inv))=
        \set{\sum_{n=-N}^{\infty}{a_{-n} X\inv[n]}:
        a_{-n} \in \FF, a_N\neq 0} \cup \set{0},
    \]
where we have the absolute value
    \[
        \abs{\sum_{n=-N}^{\infty}{a_{-n} X\inv[n]}}=3^N
    \]
for the nonzero elements, and still $\abs{0}=0$. $\FF((X\inv))$ with the given absolute value is an ultrametric space. We will restrict our attention to the unit ball in $\FF((X\inv))$ around $0$, that is the set
    \[
        \II = \set{h \in \FF((X\inv)) : \abs{h} < 1 }.   
    \]

$\II$ is the set of formal Laurent series on the form
    \[
        \sum_{n=1}^{\infty}{a_{-n} X\inv[n]}
    \]
where $a_{-n} \in \FF$, and where $0$ is the element with all the coefficients $a_{-n}=0$. We can write the absolute value on $\II$ as
    \[
        \abs{h} =
            \begin{cases}
                0 & \text{, if } h = 0, \\
                3^{- N} & \text{, if } h \neq 0,
                     N = \min \set{n : a_{-n} \neq 0}.
            \end{cases}       
    \]

For $x \in \FF((X\inv))$ we let $B(x,3^n)$ be the ball around $x$ with radius $3^n$, and for $a_{-1}, \dots , a_{-\ell} \in \FF$ we let
    \[
        B[a_{-1}, \dots , a_{-\ell}] =
        B(a_{-1}X\inv[1] + \dots + a_{-\ell}X\inv[\ell], r\inv[\ell]) \subseteq \II.
    \]
This ball consists of those elements in $\II$ with the first $\ell$ coefficients given by $a_{-1}, \dots , a_{-\ell}$.

It follows from the definition of the absolute value that every ball have radius $3\inv[n]$ for some $n$. In particular every ball inside $\II$ is of the given form. We denote the radius of the ball $B$ by $r(B)$.

In this paper we will look at the Cantor set, but in the setting of formal Laurent series. We define the 'Cantor set` as
    \[
        \CC = \set{ h \in \II : a_{-n} \in \set{0,2} }.
    \]
   
We let $\psi : \set{3^n : n \in \NN} \to \set{3\inv[r] : r \in \ZZ}$ be a function, and are going to study the set
    \[
        W_\CC(\psi) = \set{ h \in \CC :
         \abs{ h - \dfrac{g}{X^N}} < \psi(3^N),
        \text{ for infinitely many } N \in \NN, \text{ where } g \in \mathcal{F}(N) },
    \]
where
    \[
        \mathcal{F}(N) = \set{ f \in \FF[X] :
        \Coeff(f) \subseteq \set{0,2}, \deg f < N},
    \]
of $\psi$-well approximable elements in the Cantor set, by rational functions contained in the Cantor set of a specific form. In this respect, we are concerned with intrinsic Diophantine approximation.

For later we note that
    \begin{equation}
    \label{eq:1}
        \card{\mathcal{F}(N)}= 2^N,
    \end{equation}
and that
    \[
        W_\CC(\psi) = \set{ h \in \CC :
        h \in \bigcup_{g \in \mathcal{F}(N)}
                B \left( \dfrac{g}{X^N} , \psi(3^N) \right)
                \text{ for infinitely many } N \in \NN },
    \]
when expressed in terms of balls instead of approximation. So
    \[
        W_\CC(\psi)
        = \limsup_{N \to \infty} A_N
        = \set{ f \in \CC : f \in A_N \text{ for infinitely many } N \in \NN },
    \]
where
    \[
        A_N = \bigcup_{g \in \mathcal{F}(N)} B \left( \dfrac{g}{X^N}, \psi(3^N)\right).
    \]

Just as with every metric, locally compact space we can introduce the notion of Hausdorff measure, and Hausdorff dimension. We let $f:\RR_{\geq 0}\to \RR_{\geq 0}$ be a dimension function i.e. $f$ is continuous, non-decreasing and satisfy $f(0)=0$.
 We can now define the Hausdorff $f-$measure in the following manner. For $A \subseteq \FF((X\inv))$ and $\rho>0$, we let $\mathcal{B}_\rho$ be the family of countable open covers of $A$, by balls $B$ of radius $r(B)\leq \rho$.
We can now define the Hausdorff $f-$measure by
    \[
        \HH^f(A) = \lim_{\rho \to 0} \inf_{\mathcal{B} \in \mathcal{B}_\rho} \sum_{B_i \in \mathcal{B}}{f(r(B_i))}.
    \]

If $f$ is the dimension function given by $f(x)=x^s$ for a $s>0$, we call it the Hausdorff $s-$measure, and denote it by $\HH^s$. We define the Hausdorff dimension by
    \[
        \dim_H(A)=\inf\set{ s>0 : \HH^s(A)=0 }.
    \]

Using standard techniques we can determine the Hausdorff dimension of $\CC$, in fact we have the following result.
\begin{prop}\label{prop:1}
    Let $\gamma = \frac{\log(2)}{\log(3)}$. For any ball $B$ with $r(B) \leq 1$ and $B \cap \CC \neq \emptyset$ we have
        \[
            \HH^\gamma(B \cap \CC) = r(B)^\gamma,
        \]
and in particular for $B = \II$ we have
        \[
            \HH^\gamma(\CC) = 1 \text{ and } \dim_H(\CC)=\gamma.
        \]
\end{prop}

\begin{proof}
    Throughout the proof let $B$ be a ball with $B \cap \CC \neq \emptyset$ and $r(B) = 3\inv[\ell_0] \leq 1$. Then $B=B[a_{-1}, \dots, a_{-\ell_0}]$ for some $a_{-1}, \dots, a_{-\ell_0} \in \FF$.

    For the upper bound, let $\rho = 3\inv[j] \leq 3\inv[\ell_0]$. Then $B \cup \CC$ can be covered by the collection of $2^{j-\ell_0}$ balls
    \[
        \mathcal{B}'=
        \set{ B[a_{-1},\dots,a_{-\ell_0},a_{-(\ell_0+1)}, \dots, a_{-j}] :
        a_{-(\ell_0+1)}, \dots, a_{-j} \in \set{0,2}}
    \]
of radius $3\inv[j]$. We then have
    \[
        \inf_{\mathcal{B} \in \mathcal{B}_\rho} \sum_{B_i \in \mathcal{B}}{(r(B_i))^\gamma}
        \leq \sum_{B_i \in \mathcal{B'}}(r(B_i))^\gamma = 2^{j-\ell_0}(3\inv[j])^\gamma = 2\inv[\ell_0] = (3\inv[\ell_0])^\gamma = r(B)^\gamma,
    \]
since $3^\gamma = 2$. By letting $j \to \infty$, we get that $\HH^\gamma(B \cap \CC) \leq r(B)^\gamma$, which gives the upper bound.

For the lower bound let $\mathcal{B}$ be a cover of $B \cap \CC$ by balls. Then we want to show that
    \[
        r(B)^\gamma \leq \sum_{B_i \in \mathcal{B}}(r(B_i))^\gamma.
    \]

First we may restrict the balls to lie in $B$, potentially decreasing the sum. If the inequality holds true when summing over a subset of $\mathcal{B}$, then it holds true when summing over $\mathcal{B}$. Since $B \cap \CC$ is compact, due to balls in $\FF((X\inv))$ being clopen, we can cover $B \cap \CC$ by a finite subset of $\mathcal{B}$.
Furthermore if a ball $\tilde{B} = B[a_{-1}, \dots, a_{-\ell}]$ of radius $3\inv[\ell]$ have one of $a_{-1},\dots, a_{-\ell}$ equal to $1$, then $\tilde{B} \cap \CC = \emptyset$, and we remove it from the finite subcover. The remaining balls we denote by $\mathcal{B}'$, and note that $\mathcal{B}'$ is a finite cover of $B \cap \CC$, each ball having nonempty intersection with $\CC$.

Let the smallest ball in $\mathcal{B}'$ have radius $3\inv[k]$. For balls $\tilde{B} = B[a_{-1}, \dots, a_{-\ell}]$ of radius $3\inv[\ell] > 3\inv[k]$, $\tilde{B}$ can disjointly be split into three balls $A_0, A_1, A_2$ of radius $3\inv[(\ell+1)]$ by $A_i = B[a_{-1}, \dots, a_{-\ell}, i]$ for $i = 0,1,2$. Now $A_1 \cap \CC = \emptyset$ and
    \[
        r(\tilde{B})^\gamma = (3\inv[\ell])^\gamma = 3^\gamma (3\inv[(\ell +1)])^\gamma = 2 (3\inv[(\ell +1)])^\gamma = r(A_0)^\gamma+r(A_2)^\gamma,
    \]
so replacing the $B$ by $A_0$ and $A_2$ does not change the sum, and we still have a cover of $B \cap \CC$.

By iterating the procedure we end up with a cover of $B \cap \CC$ by balls of radius $3\inv[k]$. Since it is a cover we must have at least $2^{k-\ell_0}$ such balls, and hence
    \[
        \sum_{B_i \in \mathcal{B}}r(B_i)^\gamma
        \geq \sum_{B_i \in \mathcal{B}'}r(B_i)^\gamma
        \geq 2^{k-\ell_0} 3\inv[k \gamma] = 2\inv[\ell_0] = r(B)^\gamma,
    \]
which is the lower bound.
\end{proof}

We are now ready to state the analogue of the main result of \cite{MR2295506} in the setting of formal Laurent series.

\begin{thm}
    \label{thm:main}\label{THM:MAIN}
    Let $f$ be a dimension function such that $r\inv[\gamma]f(r)$ is monotonic. Then
    \[
        \HH^f(W_\CC(\psi))=
            \begin{cases}
                0 & \text{ if } \sum_{n=1}^{\infty}{f(\psi(3^n))\times (3^n)^\gamma}<\infty\\
                \HH^f(\CC) & \text{ if } \sum_{n=1}^{\infty}{f(\psi(3^n))\times (3^n)^\gamma}=\infty
            \end{cases}
        \]
\end{thm}

\section{Toolbox}
\label{sec:mtt}

In this section we collect a lot of results which we will use in the rest of the paper.

We will need the following version of the diverging part of the Borel--Cantelli lemma, Lemma 2.3 in \cite{MR1672558}.

\begin{lem}\label{BC}
    Let $(X,\mu)$ be a finite measure space. Let $\mathcal{A}_n$ be a sequence of measurable subsets of $X$. If
    \[
        \sum_{n = 1}^{\infty} \mu(\mathcal{A}_n) = \infty,
    \]
then
    \[
        \mu(\limsup_{n \to \infty} \mathcal{A}_n) \geq
        \limsup_{N \to \infty}\dfrac
        {\left( \sum_{k = 1}^{N}\mu(\mathcal{A}_k) \right)^2}
        {\sum_{n,m = 1}^{N}  \mu(\mathcal{A}_n \cap \mathcal{A}_m)}.   
    \]
\end{lem}

Furthermore we need the following generalisation of the Mass Transference Principle, Theorem 3 in \cite{MR2259250}, but slightly simplified to the current setting.

For a dimension function $f$ and a ball $B$ inside $\CC$, that is a ball in the relative topology, of the form $B=B(x,r)$, we can define the transformation of $B$ by $f$ as the ball
    \[
        B^f = B(x, f(r)^{1/\gamma}).
    \]

If the dimension function is just $r \mapsto r^s$ for some $s > 0$, we just write the transformed ball as $B^s$. In particular we have that $B^\gamma = B$.

\begin{thm}[The Generalised Mass Transference Principle]\label{GMTP}
    Let $\set{B_i}_{i \in \NN}$ be a sequence of balls in $\CC$ with $r(B_i) \to 0$ as $i \to \infty$. Let $f$ be a dimension function such that $x \mapsto x\inv[\gamma]f(x)$ is monotonic. Suppose that any ball $B \subseteq \CC$ satisfy
    \[
        \mathcal{H}^\gamma \left( B \cap \limsup_{i \to \infty} B_i^f \right)
        = \mathcal{H}^\gamma (B).
    \]
Then any ball $B \subseteq \CC$ satisfy
    \[
        \mathcal{H}^f \left( B \cap \limsup_{i \to \infty} B_i^\gamma \right)
        = \mathcal{H}^f (B).
    \]
\end{thm}

We will also need the theory of continued fractions over formal Laurent series as first studied by Artin in \cite{MR1544651}. Every rational function $\frac{g}{h}$ can be written uniquely as a finite continued fraction
    \[
        \dfrac{g}{h} = a_0 +    \cfrac{1}{a_1 + 
                                \cfrac{1}{ \ddots +
                                \cfrac{1}{a_n} }}
                    = [a_0; a_1, \dots, a_n]
    \]
with $a_0, a_1, \dots, a_n \in \FF[X]$ and $\deg(a_1), \dots, \deg(a_n) \geq 1$.
In a similar way every $x \in \FF((X\inv)) \setminus \FF(X)$ can uniquely be written as an infinite continued fraction
    \[
        x = a_0 +   \cfrac{1}{a_1 +
                    \cfrac{1}{a_2 +  
                    \cfrac{1}{ \ddots }}}
                    = [a_0; a_1, a_2, \dots]
    \]
with $a_i \in \FF[X]$ for all $i \geq 0$, and $\deg(a_i) \geq 1$ for $i \geq 1$. We call the polynomials $a_i$ the partial quotients of $x$, and the rational functions
    \[
        \dfrac{P_j}{Q_j} = [a_0; a_1, \dots, a_j]
    \]
the convergents to $x$.

Furthermore, from the ultrametric property on $\FF((X\inv))$, we have that
    \[
        \abs{x - \dfrac{P_j}{Q_j}} = \dfrac{1}{\abs{a_{j+1}} \abs{Q_j}^2}
    \]
for all the convergents.

We will also need the Folding Lemma, Proposition 2 in \cite{MR1149740}.

\begin{lem}[Folding Lemma]
    If $\frac{g}{h} = [a_0; a_1, \dots, a_n]$ is a rational function, and $t$ is a polynomial with $\deg(t) \geq 1$, then
    \[
        \dfrac{g}{h} + \dfrac{(-1)^n}{th^2} =
        [a_0; a_1, \dots, a_n, t, -a_n, \dots, -a_1].
    \]
\end{lem}

\section{Proof of Theorem \ref{thm:main}}

\subsection*{Convergent case:}

Since
    \[
        \sum_{n = 1}^{\infty}{f(\psi(3^n)) \times (3^n)^\gamma}<\infty,
    \]
and since $f$ is a dimension function, we have that $\psi(3^n) \to 0$ as $n \to \infty$.

Let $\rho > 0$ be given. Then there exists an integer $N_\rho$, such that
    \begin{equation}
    \label{eq:4}
        \psi(3^n) \leq \rho \text{ for all }n \geq N_\rho.
    \end{equation}
Furthermore we may choose $N_\rho$ such that $N_\rho\to \infty$ as $\rho \to 0$.

We can now cover $W_\CC(\psi)$ by the countable collection of balls
    \[
        W_\CC(\psi) \subseteq \bigcup_{N\geq N_\rho} A_N
            = \bigcup_{N\geq N_\rho} \bigcup_{g \in \mathcal{F}(N)}
              B \left( \dfrac{g}{X^N},\psi(3^N)\right),
    \]
each having radius $<\rho$ by \eqref{eq:4}. Hence
    \begin{align*}
        \inf_{\mathcal{B} \in \mathcal{B}_\rho}\sum_{B_i \in \mathcal{B}}    {f(r(B_i))}
            &\leq \sum_{N\geq N_\rho} \sum_{g \in \mathcal{F}(N)}     f(\psi(3^N))\\
            &= \sum_{N\geq N_\rho}f(\psi(3^N)) \times \card{\mathcal{F}(N)}\\
            &\overset{\eqref{eq:1}}{=} \sum_{N\geq N_\rho}f(\psi(3^N)) \times (3^N)^\gamma \to 0
    \end{align*}
as $\rho \to 0$. So we have that
    \[
        \HH^f(W_\CC(\psi))=0
    \]
in this case.

\subsection*{Divergent case:}

To simplify the notation we let $\mu$ be the Hausdorff $\gamma$-measure restricted to $\CC$, that is
    \[
        \mu(A)=\HH^\gamma(A \cap \CC)
    \]
for every Borel set $A$.

Furthermore we define
    \[
        W_\CC^*(\psi) = \set{ h \in \CC :
         \abs{ h - \dfrac{g}{X^N}} < \psi(3^N),
        \text{ for infinitely many } N \in \NN, \text{ where } g \in \mathcal{F}^*(N) },
    \]
where
    \[
        \mathcal{F}^*(N) = \set{ f \in \FF[X] :
        \Coeff(f) \subseteq \set{0,2}, \deg f < N \text{ and } f(0)=2}.
    \]

We note that
    \begin{equation}
    \label{eq:5}
        \card{\mathcal{F}^*(N)}= 2^{N-1},
    \end{equation}
and that just like before
    \[
        W_\CC^*(\psi) = \set{ h \in \CC :
        h \in \bigcup_{g \in \mathcal{F}^*(N)}
                B \left( \dfrac{g}{X^N} , \psi(3^N) \right)
                \text{ for infinitely many } N \in \NN },
    \]
when expressed in terms of balls instead of approximation, and
    \[
        W_\CC^*(\psi)
        = \limsup_{N \to \infty} A_N^*
        = \set{ f \in \CC : f \in A_N^* \text{ for infinitely many } N \in \NN },
    \]
where
    \[
        A_N^* = \bigcup_{g \in \mathcal{F}^*(N)} B \left( \dfrac{g}{X^N}, \psi(3^N)\right).
    \]

Proving the divergent part of the theorem, but with $W_\CC^*(\psi)$ instead of $W_\CC(\psi)$, proves the result since
    \[
        W_\CC^*(\psi) \subseteq W_\CC(\psi) \subseteq \CC,   
    \]
so we do that.

First, we prove the divergent part of the theorem in the special case when the dimension function $f$ is just the function $r \mapsto r^\gamma$, that is the following theorem:

\begin{thm}\label{thm:1}
$\mu(W_\CC^*(\psi)) = \mu(\CC) = 1$ if $\sum_{n=1}^{\infty}{(\psi(3^n)\times 3^n)^\gamma}=\infty$.
\end{thm}
\begin{proof}

The proof is divided into six steps.
\begin{itemize}
\item[i)]

Without loss of generality we may assume that
    \begin{equation}
    \label{eq:2}
        \psi(3^n)\leq 3^{-n} \text{ for all }n \in \NN.
    \end{equation}

If that was not the case, the function $\Psi$, defined by $\Psi(r)=\min\set{r\inv,\psi(r)}$, would satisfy \eqref{eq:2}. Furthermore if $\Psi(3^n)=3\inv[n]$ infinitely often, we have that
    \[
        \sum_{n=1}^{\infty}{(\Psi(3^n)\times 3^n)^\gamma}=\infty.
    \]

On the other hand if $\Psi(3^n)=3\inv[n]$ only a finite number of times,
    \[
        \sum_{n=1}^{\infty}{(\Psi(3^n)\times 3^n)^\gamma} \geq \sum_{n=N}^{\infty}{(\Psi(3^n)\times 3^n)^\gamma} = \sum_{n=N}^{\infty}{(\psi(3^n) \times 3^n)^\gamma}=\infty,
    \]
for $N$ sufficiently large. Since $W_\CC^*(\Psi) \subseteq W_\CC^*(\psi)$, we could just prove the theorem with $\Psi$ instead of $\psi$.
   
\item[ii)]

Let $g,h \in \mathcal{F}^*(n)$ be different. Then
    \[
        \abs{ \dfrac{g}{X^n} - \dfrac{h}{X^n} } =
        \abs{ \dfrac{ g - h }{X^n} } \geq
        3\inv[n] \geq \psi(3^n),    
    \]
and hence
    \[
        B \left( \dfrac{g}{X^n}, \psi(3^n)\right) \cap
        B \left( \dfrac{h}{X^n}, \psi(3^n)\right) = \emptyset
    \]
due to the ultrametric property. This implies that $A_n^*$ is a disjoint union
    \[
        A_n^* = \bigsqcup_{g \in \mathcal{F}^*(n)}
            B \left( \dfrac{g}{X^n}, \psi(3^n)\right)
    \]
for every $n \in \NN$.

\item[iii)]

For any ball $B$ with $r(B) = 3\inv[\ell] \leq 1$ and $B \cap \CC \neq \emptyset$, if $n > \ell$, then
    \begin{equation}\label{eq:6}
        \card{ \set{ g \in \mathcal{F}^*(n) :
            B \left( \dfrac{g}{X^n}, \psi(3^n)\right) \subseteq B } }    = 2^{n-\ell -1}.
    \end{equation}
   
This follows since any polynomial $g \in \mathcal{F}^*(n)$ has the coefficient $a_0 = 2$ and coefficients $a_{n-1}, \dots, a_{1}$ either $0$ or $2$. The requirement that the ball $B \left( \frac{g}{X^n}, \psi(3^n)\right)$ is contained in $B$ fixes the coefficients $a_{n-1}, \dots, a_{n-\ell}$. The remaining $n- \ell -1$ coefficients can be either $0$ or $2$ giving $2^{n-\ell-1}$ elements in the set.

\item[iv)]

We can now, under the assumptions of iii), compute
    \begin{align*}
        \mu( B \cap A_n^*) &= \mu\left( \bigsqcup_{g \in \mathcal{F}^*(n)}
            B \cap B \left( \dfrac{g}{X^n}, \psi(3^n)\right) \right)\\            
                &=  \sum_{g \in \mathcal{F}^*(n)}
            \mu\left( B \cap B \left( \dfrac{g}{X^n}, \psi(3^n)\right) \right)\\           
                &=  \sum_{\substack{ g \in \mathcal{F}^*(n) \\ 
                    B \left( \frac{g}{X^n}, \psi(3^n)\right) \subseteq B }}
            \mu\left( B \left( \dfrac{g}{X^n}, \psi(3^n)\right) \right)\\
                &=   \sum_{\substack{ g \in \mathcal{F}^*(n) \\ 
                    B \left( \frac{g}{X^n}, \psi(3^n)\right) \subseteq B }}
            \psi(3^n)^\gamma\\
                &\overset{\eqref{eq:6}}{=} 2^{n-\ell-1} \psi(3^n)^\gamma                    
    \end{align*}
where we have used Proposition \ref{prop:1}. Since
    \[
        2^{n-\ell-1} = (3^n)^\gamma \times (3\inv[\ell])^\gamma    \times 3\inv[\gamma]
    \]
we have that
    \begin{equation}
    \label{eq:3}
        \mu( B \cap A_n^*)
        = r(B)^\gamma \times (\psi(3^n) \times 3^n)^\gamma \times 3\inv[\gamma].
    \end{equation}

For $B = \II$ it follows that
    \begin{equation}\label{eq:9}
        \mu( A_n^*) = (\psi(3^n) \times 3^n)^\gamma \times 3\inv[\gamma]
    \end{equation}
for all $n \in \NN$, and hence
    \begin{equation}\label{eq:8}
        \sum_{n = 1}^{\infty}
        \mu( A_n^*) = \infty.
    \end{equation}

\item[v)]

We have the following quasi-independence result
\begin{prop}
    For $n > m$ we have
        \begin{equation}
        \label{eq:7}
            \mu( A_m^* \cap A_n^*) \leq \mu( A_m^*) \mu( A_n^*)       
        \end{equation}
       
\end{prop}

\begin{proof}

Let $\psi(3^m)=3\inv[\ell]$. If $n \leq \ell$, then $3\inv[n] \geq 3\inv[\ell]=\psi(3^m)$ and from \eqref{eq:2} we have $3\inv[n] \geq \psi(3^n)$. Let $g \in \mathcal{F}^*(n), h \in \mathcal{F}^*(m)$. Since $g - h X^{n-m}$ evaluated in $0$ is $2$, we have that $g - h X^{n-m} \neq 0$, and hence
    \[
        \abs{ \dfrac{g}{X^n} - \dfrac{h}{X^m}} =
        \abs{ \dfrac{g-hX^{n-m}}{X^n} }    \geq 3\inv[n].
    \]

From this we get that
    \[
        B \left( \dfrac{g}{X^n}, \psi(3^n)\right) \cap
        B \left( \dfrac{h}{X^m}, \psi(3^m)\right) = \emptyset,
    \]
and by definition of $A_n^*$ and $A_m^*$ we then have
    \[
        A_n^* \cap A_m^* = \emptyset,   
    \]
which implies that
    \[
        \mu ( A_n^* \cap A_m^* ) = 0
    \]
and the quasi-independence is trivially satisfied.

If $n > \ell$ we have
    \begin{align*}
        \mu (A_m^* \cap A_n^*) &=
            \mu \left(
                \bigsqcup_{g \in \mathcal{F}^*(m)} \left(
                    B \left( \dfrac{g}{X^m}, \psi(3^m) \right)
                \cap A_n^* \right)
            \right)\\
        &=  \sum_{g \in \mathcal{F}^*(m)}
            \mu \left(
                    B \left( \dfrac{g}{X^m}, \psi(3^m) \right)
                \cap A_n^* 
            \right)\\
        &\overset{\eqref{eq:3}}{=}  \sum_{g \in \mathcal{F}^*(m)}
            (\psi(3^m))^\gamma \times (\psi(3^n) \times 3^n)^\gamma
                               \times 3\inv[\gamma]    \\
        &\overset{\eqref{eq:5}}{=}  2^{m-1} \times (\psi(3^m))^\gamma \times
            (\psi(3^n) \times 3^n)^\gamma \times 3\inv[\gamma]\\
        &=  \Big( 3\inv[\gamma] \times (\psi(3^m) \times 3^m)^\gamma \Big) \times
            \Big( 3\inv[\gamma] \times (\psi(3^n) \times 3^n)^\gamma \Big) \\
        &\overset{\eqref{eq:9}}{=}  \mu ( A_m^* ) \times \mu(A_n^*).
    \end{align*}

This concludes the proof of the quasi-independence.
\end{proof}

\item[vi)]

From \eqref{eq:8} we can use Lemma \ref{BC} to get

    \[
        \mu(W_\CC^*(\psi)) \geq
        \limsup_{N \to \infty}\dfrac
        {\left( \sum_{k = 1}^{N} \mu(A_k^*) \right)^2}
        {\sum_{n,m = 1}^{N}  \mu(A_n^* \cap A_m^*)} \geq
        \limsup_{N \to \infty} 1 = 1,
    \]
where the last inequality comes from the quasi-independence. Since we trivially have
    \[
        1 = \mu(\CC) \geq \mu(W_\CC^*(\psi)),
    \]
the result follows.   
\end{itemize}
\end{proof}

We will now deduce the diverging part of Theorem \ref{thm:main} from Theorem \ref{thm:1} by a standard application of the Mass Transference Principle.

\begin{proof}
Without loss of generality we will assume that $\psi(3^n) \to 0$ when $n \to \infty$, since else $W_\CC^*(\psi) = \CC$ and the result is clear. By assumption we have that
    \[
        \sum_{n = 1}^{\infty}
        f(\psi(3^n)) \times (3^n)^\gamma = \infty
    \]
and $r \mapsto r\inv[\gamma] f(r)$ is monotonic. Define $\theta$ by $\theta(r) = \lceil f(\psi(r))^{1/\gamma} \rceil_3$, where $\lceil \cdot \rceil_3$ is the function that rounds up to the nearest power of 3.

Then
    \[
        \sum_{n = 1}^{\infty}
        (\theta(3^n) \times 3^n)^\gamma = \infty,
    \]
and from Theorem \ref{thm:1} we get that $\mu(W_\CC^*(\theta)) = \mu(\CC) = 1$. This in turn implies that
    \[
        \mu(B \cap W_\CC^*(\theta)) = \mu(B \cap \CC)
    \]
for any ball $B \subseteq \CC$. Now Theorem \ref{GMTP} gives
    \[
        \HH^f(B \cap W_\CC^*(\psi)) = \HH^f(B \cap \CC)
    \]
for any ball $B \subseteq \CC$. In particular for $B = \CC$ we get the desired result.   
\end{proof}

\section{Irrationality exponent}
For an element $\xi \in \FF((X\inv))$ we define the irrationality exponent of $\xi$ as
    \[
        \tau(\xi) = \sup \set{ \tau : \abs{ \xi - \dfrac{g}{h} } < \abs{h}\inv[\tau]
            \text{ for infinitely many } \dfrac{g}{h} \in \FF(X)}.
    \]

From Dirichlet's theorem in the field of formal Laurent series we get that $\tau(\xi) \geq 2$ for all $\xi \in \FF((X\inv))$.

Furthermore for $\psi : \set{3^n:n \in \NN} \to \RR_+$ a non-increasing function we define
    \[
        \mathcal{K}(\psi) = \set{ \xi \in \II :
         \abs{ \xi - \dfrac{g}{h}} < \psi(\abs{h}),
        \text{ for infinitely many } \dfrac{g}{h} \in \FF(X) }.
    \]
   
We now have the following theorem.
\begin{thm}\label{thm:2}
    Let $\psi : \set{3^n:n \in \NN} \to \RR_+$ be a non-increasing function such that \break$x \mapsto x^2 \psi(x)$ is non-increasing and tends to $0$ as $3^n$ tends to infinity. For any $c \in (0,\frac{1}{3})$ the set
    \[
        \mathcal{K}(\psi) \setminus \mathcal{K}(c\psi) \cap \CC
    \]
is uncountable.
\end{thm}

From Theorem \ref{thm:2} we get the following result.

\begin{cor}
    For any $\tau \in [2,\infty]$ there exist uncountably many elements in $\CC$ with irrationality exponent $\tau$.
\end{cor}
\begin{proof}
    For $\tau \in (2,\infty)$ we can use Theorem \ref{thm:2} with $\psi(x) = x\inv[\tau]$. For $\tau = 2$ we can use the function $\psi(x) = (x \log x)\inv[2]$. Finally for $\tau = \infty$ the element
    \[
        \sum_{n = 1}^{\infty} 2 X^{-n!}   
    \]
has the desired irrationality exponent, since the proof by Liouville for the corresponding real case can be applied. In a similar way we can construct uncountably many with irrationality exponent $\tau = \infty$.
\end{proof}

\begin{proof}[Proof of Theorem \ref{thm:2}]
    Let $u_1 , v_1 = 1$ and define recursively $u_{i+1}$ as the integer satisfying
    \[
        1 < 3^{u_{i+1}} 3^{2v_i} \psi(3^{v_i}) \leq 3,
    \]
    and $v_{i+1}$ by
    \[
        v_{i+1}=u_{i+1}+2v_i.
    \]

From the assumption on $x \mapsto x^2 \psi(x)$ we get that the sequence $\set{u_i}_{i \in \NN}$ is non-decreasing and tends to infinity as $i \to \infty$.

From  $\set{u_i}_{i \in \NN}$ we now construct the following sequence of rational functions:

Let
\begin{align*}
    \xi_{\mathbf{u},1} &= [0;-X^{u_1}] = \dfrac{-1}{X^{u_1}} = \dfrac{P_1}{X^{v_1}} \\
    \xi_{\mathbf{u},2} &= [0;-X^{u_1}, X^{u_2}, X^{u_1}]
            = \dfrac{-1}{X^{u_1}}+\dfrac{-1}{X^{u_2+2u_1}} = \dfrac{P_2}{X^{v_2}} \\
    \xi_{\mathbf{u},3} &= [0;-X^{u_1}, X^{u_2}, X^{u_1}, X^{u_3}, -X^{u_1}, -X^{u_2}, X^{u_1}]
            = \dfrac{-1}{X^{v_1}}+\dfrac{-1}{X^{v_2}}+\dfrac{-1}{X^{v_3}} = \dfrac{P_3}{X^{v_3}} \\
    \vdots
\end{align*}

where the element $\xi_{\mathbf{u},n+1}$ is constructed from $\xi_{\mathbf{u},n}$ by applying the Folding Lemma. Since we are in characteristic $3$, and $\set{v_i}_{i \in \NN}$ is strictly increasing, each of the rational functions is in $\CC$. They converge to the element $\xi_{\mathbf{u},\infty} \in \CC$. Furthermore by construction each of the rational functions $\xi_{\mathbf{u},n}$ is a convergent of $\xi_{\mathbf{u},\infty}$.

We have that
    \[
        \abs{\xi_{\mathbf{u},\infty} - \xi_{\mathbf{u},n}} = \dfrac{1}{3^{v_{n+1}}}
             < \psi(3^{v_n}),
    \]
and hence $\xi_{\mathbf{u},\infty} \in \mathcal{K}(\psi)$.

For the other part it is sufficient to show that all the convergents $\frac{P_j}{Q_j}$ satisfy
    \[
        \abs{\xi_{\mathbf{u},\infty} - \dfrac{P_j}{Q_j}} > c \psi(\abs{Q_j})
    \]
as the convergents are best approximants.

For $2^{i-1} \leq j < 2^i$ we have that $\abs{a_{j+1}} \leq 3^{u_{i+1}}$. Now
    \[
        \abs{\xi_{\mathbf{u},\infty} - \dfrac{P_j}{Q_j}}
            = \dfrac{1}{\abs{a_{j+1}} \abs{Q_j}^2}
    \]
but since
    \[
        \abs{a_{j+1}} \abs{Q_j}^2 \psi(\abs{Q_j})
            \leq 3^{u_{i+1}} \abs{Q_j}^2 \psi(\abs{Q_j})
            \leq 3^{u_{i+1}} 3^{2v_i} \psi(3^{v_i})\leq 3
    \]
we have
    \[
        \abs{\xi_{\mathbf{u},\infty} - \dfrac{P_j}{Q_j}}
            \geq \dfrac{\psi({\abs{Q_j}})}{3}
            > c \psi({\abs{Q_j}})
    \]
and hence $\xi_{\mathbf{u},\infty} \not \in \mathcal{K}(c\psi)$.

In order to get uncountable many elements with the desired property, the sequence $\set{u_i}_{i \in \NN}$ can be modified in the following way. Define $\set{u'_i}_{i \in \NN}$ by $u'_1 = 1, u'_{2n} = u_n$ and $u'_{2n+1} \in \set{1,2}$. By the same proof we get that each of the formal Laurent series $\xi_{\mathbf{u'},\infty} \in \mathcal{K}(\psi) \setminus \mathcal{K}(c\psi) \cap \CC$, and since there is uncountably many such sequences, each giving different formal Laurent series, we have the desired result.
\end{proof}

\section{Concluding remarks}
\label{sec:concluding-remarks}

Let $p$ be a prime, $q = p^n$ for some $n \geq 1$, and $\FFq$ the field with $q$ elements. We can from $\FFq$ construct the polynomials $\FFq[X]$ and the rational functions $\FFq(X)$ with absolute value $\abs{\frac{g}{h}} = q^{\deg g - \deg h}$ for the non-zero rational functions, and $\abs{0} = 0$. Completing with respect to this absolute value gives the formal Laurent series over $\FFq$.

Like before we restrict ourselves to the unit ball, that is elements of the form
    \[
        \sum_{n = 1}^{\infty} a_{-n} X\inv[n] , a_{-n} \in \FFq.
    \]

Let $\mathcal{A} \subseteq \FFq$ with $2 \leq \card{\mathcal{A}} < q$, and construct the missing digit set
    \[
        \MDSA = \set{ \sum_{n = 1}^{\infty} a_{-n} X\inv[n] :
                     a_{-n} \in \mathcal{A} }.
    \]
In the particular the case $q = 3$ and $\mathcal{A} = \set{0, 2}$ we just have $\MDSA = \CC$.

The results of this paper also holds true in the more general setting of missing digit sets, as the proofs can be modified to this situation. We have that the Hausdorff  dimension of $\MDSA$ is $\gamma_\mathcal{A} = \frac{\log \card{\mathcal{A}}}{\log q}$ with $\HH^{\gamma_\mathcal{A}}(\MDSA) = 1$. Furthermore for a function $\psi : \set{ q^n : n \in \NN} \to \set{ q\inv[r] : r \in \ZZ}$ we define the set $W_{\MDSA}(\psi)$ by
    \[
         \set{ h \in \MDSA :
         \abs{ h - \dfrac{g}{X^N}} < \psi(q^N),
        \text{ for infinitely many } N \in \NN,
        \text{ where } g \in \mathcal{F}_\mathcal{A}(N) },
    \]
where
    \[
        \mathcal{F}_\mathcal{A}(N) = \set{ f \in \FFq[X] :
        \Coeff(f) \subseteq \mathcal{A}, \deg f < N},
    \]
we have the following theorem.
\begin{thm}
    Let $f$ be a dimension function such that $r\inv[\gamma_\mathcal{A}]f(r)$ is monotonic. Then
    \[
        \HH^f(W_{\MDSA}(\psi))=
            \begin{cases}
                0 & \text{ if } \sum_{n=1}^{\infty}{f(\psi(q^n))\times (q^n)^{\gamma_\mathcal{A}}}<\infty\\
                \HH^f(\MDSA) & \text{ if } \sum_{n=1}^{\infty}{f(\psi(q^n))\times (q^n)^{\gamma_\mathcal{A}}}=\infty
            \end{cases}
        \]
\end{thm}

Finally the results about irrationality exponents also hold true in the more general setting. For an element $\xi \in \FFq((X\inv))$ we define the irrationality exponent in the same way as before as
    \[
        \tau(\xi) = \sup \set{ \tau : \abs{ \xi - \dfrac{g}{h} } < \abs{h}\inv[\tau]
            \text{ for infinitely many } \dfrac{g}{h} \in \FFq(X)}.
    \]

Furthermore for $\psi : \set{q^n:n \in \NN} \to \RR_+$ a non-increasing function and $\II$ the unit ball in $\FFq((X\inv))$ we define
    \[
        \mathcal{K}(\psi) = \set{ \xi \in \II :
         \abs{ \xi - \dfrac{g}{h}} < \psi(\abs{h}),
        \text{ for infinitely many } \dfrac{g}{h} \in \FFq(X) },
    \]
and we have
\begin{thm}
    Assume that $x \mapsto x^2 \psi(x)$ is non-increasing and tends to $0$ as $q^n$ tends to infinity. For any $c \in (0,\frac{1}{q})$ the set
    \[
        \mathcal{K}(\psi) \setminus \mathcal{K}(c\psi) \cap \MDSA
    \]
is uncountable.
\end{thm}
And the corresponding corollary
\begin{cor}
    For any $\tau \in [2,\infty]$ there exist uncountably many elements in $\MDSA$ with irrationality exponent $\tau$.
\end{cor}

\section*{Acknowledgements}

I would like to thank my Ph.D. advisor Simon Kristensen for directing me toward this problem, and for helpful suggestions along the way.

\providecommand{\bysame}{\leavevmode\hbox to3em{\hrulefill}\thinspace}
\providecommand{\MR}{\relax\ifhmode\unskip\space\fi MR }
\providecommand{\MRhref}[2]{%
  \href{http://www.ams.org/mathscinet-getitem?mr=#1}{#2}
}
\providecommand{\href}[2]{#2}


\end{document}